\documentclass{amsart}

\usepackage{amssymb,amsmath,amsthm}
\usepackage{color}
\makeindex
\usepackage{enumerate}
\usepackage{dsfont}
\usepackage{mathtools}
\usepackage{mathrsfs}

\newtheorem{thm}{Theorem}
\newtheorem{theorem}[thm]{Theorem}

\newtheorem{lemma}[thm]{Lemma}
\newtheorem{proposition}[thm]{Proposition}

\newtheorem*{thm:evalconverge}{Theorem \ref{thm:evalconverge}}
\newtheorem*{thm:lsiboundRn}{Theorem \ref{thm:lsiboundRn}}

\begin{document}

\title[]{Elementary proof of logarithmic Sobolev inequalities for Gaussian convolutions on $\mathbb{R}$}

\author[]{David Zimmermann}
\address{Department of Mathematics\\
  University of California \\
  San Diego 92093}
\email{dszimmer@math.ucsd.edu}
\maketitle

\begin{abstract}
In a 2013 paper, the author showed that the convolution of a compactly supported measure on the real line with a Gaussian measure satisfies a logarithmic Sobolev inequality (LSI). In a 2014 paper, the author gave bounds for the optimal constants in these LSIs. In this paper, we give a simpler, elementary proof of this result.
\end{abstract}

\section{Introduction}

A probability measure $\mu$ on $\mathbb{R}^n$ is said to satisfy a logarithmic Sobolev inequality (LSI) with constant $c\in \mathbb{R}$ if 
\begin{equation*}\label{eq:lsi}
\mathrm{Ent}_\mu(f^2)\leq c\ \mathscr{E}(f,f)
\end{equation*}
for all locally Lipschitz functions $f:\mathbb{R}^n\rightarrow\mathbb{R}_+$, where $\mathrm{Ent}_\mu$, called the entropy functional, is defined as
\begin{align*}
\mathrm{Ent}_\mu (f)\coloneqq&\int f\log\frac{f}{\int f\ d\mu}\,d\mu
\end{align*}
and $\mathscr{E}(f,f)$, the energy of $f$, is defined as
$$
\mathscr{E}(f,f)\coloneqq\int|\nabla f|^2 d\mu,
$$
with $|\nabla f|$ defined as
$$
|\nabla f|(x)\coloneqq\limsup_{y\rightarrow x}\frac{|f(x)-f(y)|}{|x-y|}
$$
so that $|\nabla f|$ is defined everywhere and coincides with the usual notion of gradient where $f$ is differentiable.
The smallest $c$ for which a LSI with constant $c$ holds is called the optimal log-Sobolev constant for $\mu$.

LSIs are a useful tool that have been applied in various areas of mathematics, cf. \cite{Ba94, Ba97, BH97, BT06, Da87, Da90, DS84, DS96, GR98, GZ03, Le96, Le01, Le03, Vi03, Ya96, Ya97, Ze92}. In \cite{Zi13}, the present author showed that the convolution of a compactly supported measure on $\mathbb{R}$ with a Gaussian measure satisfies a LSI, and an application of this fact to random matrix theory was given. In \cite[Thms. 2,3]{Zi14}, bounds for the optimal constants in these LSIs were given, and the results were extended to $\mathbb{R}^n$. Those results are stated as Theorems \ref{thm:lsibound} and \ref{thm:lsiboundRn} below. (See \cite{WW13} for statements about LSIs for convolutions with more general measures).

\begin{theorem}\label{thm:lsibound}
Let $\mu$ be a probability measure on $\mathbb{R}$ whose support is contained in an interval of length $2R$, and let $\gamma_\delta$ be the centered Gaussian of variance $\delta>0$, i.e., $d\gamma_\delta(t)=(2\pi\delta)^{-1/2}\exp(-\frac{t^2}{2\delta})dt$. Then for some absolute constants $K_i$, the optimal log-Sobolev constant $c(\delta)$ for $\mu*\gamma_\delta$ satisfies
$$
c(\delta)\leq K_1\,\frac{\delta^{3/2}R}{4R^2+\delta}\exp\left(\frac{2R^2}{\delta}\right)+K_2\,(\sqrt\delta+2R)^2.
$$
In particular, if $\delta\leq R^2$, then
$$
c(\delta)\leq K_3\, \frac{\delta^{3/2}}{R}\exp\left(\frac{2R^2}{\delta}\right).
$$
The $K_i$ can be taken in the above inequalities to be $K_1=6905, K_2=4989, K_3=7803$. 
\end{theorem}

\begin{theorem}\label{thm:lsiboundRn}
Let $\mu$ be a probability measure on $\mathbb{R}^n$ whose support is contained in a ball of radius $R$, and let $\gamma_\delta$ be the centered Gaussian of variance $\delta$ with $0<\delta\leq R^2$, i.e., $d\gamma_\delta(x)=(2\pi\delta)^{-n/2}\exp(-\frac{|x|^2}{2\delta})dx$. Then for some absolute constant $K$, the optimal log-Sobolev constant $c(\delta)$ for $\mu*\gamma_\delta$ satisfies
$$
c(\delta)\leq K\, R^2\exp\left(20n+\frac{5R^2}{\delta}\right).
$$
$K$ can be taken above to be $289$. 
\end{theorem}

Theorem \ref{thm:lsibound} was proved in \cite{Zi14} using the following theorem due to Bobkov and G\"{o}tze \cite[p.25, Thm 5.3]{BG99}:

\begin{theorem}[Bobkov, G\"{o}tze]\label{thm:bg}
 Let $\mu$ be a Borel probability measure on $\mathbb{R}$ with distribution function
$F(x)=\mu((-\infty,x])$. Let $p$ be the density of the absolutely continuous part of $\mu$
with respect to Lebesgue measure, and let $m$ be a median of $\mu$. Let
\begin{align*}
D_0&=\sup_{x<m}\left(F(x)\cdot\log\frac{1}{F(x)}\cdot\int_x^m\frac{1}{p(t)}dt\right),\\
D_1&=\sup_{x>m}\left((1-F(x))\cdot\log\frac{1}{1-F(x)}\cdot\int_m^x\frac{1}{p(t)}dt\right),
\end{align*}
defining $D_0$ and $D_1$ to be zero if $\mu((-\infty,m))=0$ or  $\mu((m,\infty))=0$, respectively, and using the convention $0\cdot\infty=0$.
Then the optimal log Sobolev constant $c$ for $\mu$ satisfies $\frac{1}{150}(D_0+D_1)\leq c\leq 468(D_0+D_1)$.
\end{theorem}

Theorem \ref{thm:lsiboundRn} was proved in \cite{Zi14} using the following theorem due to Cattiaux, Guillin, and Wu \cite[Thm. 1.2]{CGW10}:

\begin{theorem}[Cattiaux, Guillin, Wu]\label{thm:CGW}
Let $\mu$ be a probability measure on $\mathbb{R}^n$ with $d\mu(x)=e^{-V(x)}dx$ for some $V\in C^2(\mathbb{R}^n)$. Suppose the following:
\begin{enumerate}
\item\label{assum:hess}
There exists a constant $K\leq 0$ such that $\mathrm{Hess}(V)\geq K I$.

\item\label{assum:lyapunov}
There exists a $W\in C^2(\mathbb{R}^n)$ with $W\geq 1$ and constants $b,c>0$ such that
$$
\Delta W(x)-\langle\nabla V,\nabla W\rangle(x)\leq(b-c|x|^2)W(x)
$$
for all $x\in\mathbb{R}^n$.
\end{enumerate}
Then $\mu$ satisfies a LSI.
\end{theorem}

The goal of the present paper is to provide an elementary proof of Theorem \ref{thm:lsibound}. The result proved is the following:

\begin{theorem}\label{thm:elementary}
Let $\mu$ be a probability measure on $\mathbb{R}$ whose support is contained in an interval of length $2R$, and let $\gamma_\delta$ be the centered Gaussian of variance $\delta>0$, i.e., $d\gamma_\delta(t)=(2\pi\delta)^{-1/2}\exp(-\frac{t^2}{2\delta})dt$. Then the optimal log-Sobolev constant $c(\delta)$ for $\mu*\gamma_\delta$ satisfies
$$
c(\delta)\leq\max\left(2\delta\exp\left(\frac{4R^2}{\delta}+\frac{4R}{\sqrt\delta}+\frac{1}{4}\right),2\delta\exp\left(\frac{24 R^2}{\delta}\right)\right).
$$
In particular, if $\delta\leq 16R^2$, we have
$$
c(\delta)\leq2\delta\exp\left(\frac{24 R^2}{\delta}\right).
$$
\end{theorem}

The bound in Theorem \ref{thm:elementary} is worse than the bound in Theorem \ref{thm:lsibound} for small $\delta$, but still has an order of magnitude that is exponential in $R^2/\delta$. (It is shown in  \cite[Example 21]{Zi14} that one cannot do better than exponential in $R^2/\delta$ for small $\delta$.) 

\section{Proof of Theorem \ref{thm:elementary}}

The proof of Theorem \ref{thm:elementary} is based on two facts: first, the Gaussian measure $\gamma_1$ of unit variance satisfies a LSI with constant $2$. Second, Lipshitz functions preserve LSIs. We give a precise statement of this second fact below.

\begin{proposition}\label{prop:lip}
Let $\mu$ be a measure on $\mathbb{R}$ that satisfies a LSI with constant $c$, and let $T:\mathbb{R}^n\rightarrow\mathbb{R}^n$ be Lipschitz. Then the push-forward measure $T_*\mu$ also satisfies a LSI with constant $c||T||_{\mathrm{Lip}}^2$.
\end{proposition}

\begin{proof}
Let $g:\mathbb{R}^n\rightarrow\mathbb{R}$ be locally Lipschitz. Then $g\circ T$ is locally Lipschitz, so by the LSI for $\mu$,
\begin{equation}\label{eq:lem}
\int (g\circ T)^2\log\frac{(g\circ T)^2}{\int (g\circ T)^2\ d\mu}\,d\mu\leq c\int|\nabla (g\circ T)|^2 d\mu.
\end{equation}
But since $T$ is Lipschitz,
$$
|\nabla (g\circ T)| \leq (|\nabla g|\circ T)||T||_{\mathrm{Lip}}.
$$
So by a change of variables, (\ref{eq:lem}) simply becomes
$$
\int g^2\log\frac{g^2}{\int g^2\ dT_*\mu}\,dT_*\mu\leq c||T||_{\mathrm{Lip}}^2\int|\nabla g|^2 dT_*\mu.
$$
as desired. 
\end{proof}

We now prove Theorem \ref{thm:elementary}.

\begin{proof}[Proof of Theorem \ref{thm:elementary}]
In light of Proposition \ref{prop:lip}, we will establish the theorem by showing that $\mu*\gamma_\delta$ is the push-forward of $\gamma_1$ under a Lipschitz map. By translation invariance of LSI, we can assume that $\mathrm{supp}(\mu)\subseteq[-R,R]$. We will also first assume that $\delta=1$ (the general case will be handled at the end of the proof by a scaling argument).

Let $F$ and $G$ be the cumulative distribution functions of $\gamma_1$ and $\mu*\gamma_1$, i.e.,
$$
F(x)=\int_{-\infty}^x p(t)\,dt,\qquad G(x)=\int_{-\infty}^x q(t)\,dt,
$$
where
$$
p(t)=\frac{1}{\sqrt{2\pi}}\exp\left(-\frac{t^2}{2}\right) \qquad \mbox{and}\qquad q(t)=\int_{-R}^R p(t-s)\,d\mu(s).
$$

Notice that $q$ is smooth and strictly positive, so that $G^{-1}\circ F$ is well-defined and smooth. It is readily seen
that $(G^{-1}\circ F)_*(\gamma_1)=\mu*\gamma_1$, so to establish the theorem we simply need to bound the derivative of $G^{-1}\circ F$.

Now
$$
(G^{-1}\circ F)'(x)=\frac{1}{G'((G^{-1}\circ F)(x))}\cdot F'(x)=\frac{p(x)}{q((G^{-1}\circ F)(x))}.
$$  
We will bound the above derivative in cases -- when $x\geq 2R$, when $-2R\leq x\leq 2R$, and when $x\leq -2R$.

We first consider the case $x\geq 2R$. Define
$$
\Lambda(x)=\int_{-R}^{R}e^{xs}d\mu(s), \qquad K(x)=\frac{\log \Lambda(x)+R}{x}.
$$ 
Note $\Lambda$ and $K$ are smooth for $x\neq 0$. 

\begin{lemma}\label{lem:boundqbyp}
For $x\geq2R$,
$$
\exp\left(-2R^2-2R-\frac{1}{8}\right)p(x)\leq q(x+K(x))\leq e^{-R}\,p(x).
$$
\end{lemma}

\begin{proof}
By definition of $q, p, \Lambda$, and $K$,
\begin{align*}
q(x+K(x))=\int_{-R}^{R}p(x+K(x)-s)\,d\mu(s)=&\,p(x)\cdot e^{-xK(x)}\int_{-R}^{R}\exp\left(-\frac{(K(x)-s)^2}{2}\right)\cdot e^{xs}\,d\mu(s)\\
=&\,\frac{e^{-R}\,p(x)}{\Lambda(x)}\int_{-R}^{R}\exp\left(-\frac{(K(x)-s)^2}{2}\right)\cdot e^{xs}\,d\mu(s)\\
\leq&\,\frac{e^{-R}\,p(x)}{\Lambda(x)}\int_{-R}^{R}e^{xs}\,d\mu(s)\\
=&\,e^{-R}\,p(x).
\end{align*}
To get the other inequality, first note that $e^{-Rx}\leq\Lambda(x)\leq e^{Rx}$. (These are just the maximum and minimum values in the integrand defining $\Lambda$.) This implies that $-R+R/x\leq K(x)\leq R+R/x,$ so for $-R\leq s\leq R$ and $x\geq 2R$, we have
$$
-2R-\frac{R}{x}\leq-2R+\frac{R}{x}\leq K(x)-s\leq 2R+\frac{R}{x}
$$
so that
$$
\exp\left(-\frac{(K(x)-s)^2}{2}\right)\geq\exp\left(-\frac{(2R+R/x)^2}{2}\right)\geq\exp\left(-\frac{(2R+R/(2R))^2}{2}\right)=\exp\left(-2R^2-R-\frac{1}{8}\right).
$$
Therefore
\begin{align*}
q(x+K(x))=&\,\frac{e^{-R}\,p(x)}{\Lambda(x)}\int_{-R}^{R}\exp\left(-\frac{(K(x)-s)^2}{2}\right)\cdot e^{xs}\,d\mu(s)\geq \exp\left(-2R^2-2R-\frac{1}{8}\right)p(x).
\end{align*}

\end{proof}

\newpage

\begin{lemma}\label{lem:k'}
$K'(x)\leq R$ for $x\geq 2R$.
\end{lemma}
\begin{proof}
Recall that $e^{-Rx}\leq\Lambda(x)$. (Again, $e^{-Rx}$ is the minimum value in the integrand defining $\Lambda$). We therefore have
\begin{align*}
K'(x)=\frac{\Lambda'(x)}{x\Lambda(x)}-\frac{\log\Lambda(x)}{x^2}-\frac{R}{x^2}=&\frac{\int_{-R}^R s\,e^{sx}\,d\mu(s)}{x\Lambda(x)}-\frac{\log\Lambda(x)}{x^2}-\frac{R}{x^2}\\
\leq&\frac{R\int_{-R}^R e^{sx}\,d\mu(s)}{x\Lambda(x)}+\frac{Rx}{x^2}-\frac{R}{x^2}\\
=&\frac{2R}{x}-\frac{R}{x^2}.
\end{align*}
By elementary calculus, the above has a maximum value of $R$.
\end{proof}

\begin{lemma}\label{lem:cdfbound}
For $x\geq 2R$,
$$
x-R\leq(G^{-1}\circ F)(x)\leq x+K(x).
$$
\end{lemma}

\begin{proof}
Since $G$ and $G^{-1}$ are increasing, the lemma is equivalent to
$$
G(x-R)\leq F(x)\leq G(x+K(x)).
$$
The first inequality follows from the definition of $G$ and the Fubini-Tonelli Theorem:
\begin{align*}
G(x-R)=\int_{-\infty}^{x-R} q(t)\,dt=\int_{-\infty}^x\int_{-R}^R p(t-s)\,d\mu(s)\,dt=&\int_{-R}^R \int_{-\infty}^{x-R} p(t-s)\,dt\,d\mu(s)\\
=&\int_{-R}^R \int_{-\infty}^{x-R+s} p(u)\,du\,d\mu(s)\\
&\mbox{where }u=t-s\\
\leq&\int_{-R}^R \int_{-\infty}^x p(u)\,dt\,d\mu(s)\\
=&F(x).
\end{align*}

To establish the other inequality, we use Lemmas \ref{lem:boundqbyp} and \ref{lem:k'}:
\begin{align*}
1-G(x+K(x))=\int_{x+K(x)}^\infty q(t)\,dt=&\int_{x}^\infty q(u+K(u))(1+K'(u))\,du\\
&\mbox{where }t=u+K(u)\\
\leq&\int_{x}^\infty p(u)e^{-R}(1+R)\,du\\
&\mbox{by Lemmas \ref{lem:boundqbyp} and \ref{lem:k'}}\\
\leq&\int_{x}^\infty p(u)\,du\\
&\mbox{since }e^R\geq 1+R\\
=&\,1-F(x),
\end{align*}
so that $F(x)\leq G(x+K(x))$, as desired.
\end{proof}

We are almost ready to bound $(G^{-1}\circ F)'(x)$ for $x\geq 2R$. The last observation to make is that $q$ is decreasing on $[R,\infty)$ since

\begin{align*}
q'(t)=\int_{-R}^R p'(t-s)\,d\mu(s)=\int_{-R}^R-(t-s)p(t-s)\,d\mu(s)\leq 0 \qquad\mbox{for }t\geq R.
\end{align*}

So for $x\geq 2R$ we have, by lemma \ref{lem:cdfbound},
$$
q((G^{-1}\circ F)(x))\geq q(x+K(x)).
$$

Combining this with Lemma \ref{lem:boundqbyp}, we get
\begin{align*}
(G^{-1}\circ F)'(x)=\frac{p(x)}{q((G^{-1}\circ F)(x))}\leq\frac{p(x)}{q(x+K(x))}\leq\exp\left(2R^2+2R+\frac{1}{8}\right)
\end{align*}
for $x\geq 2R$.

In the case where $-2R\leq x\leq 2R$, first note that for all $x$,
$$
x-R\leq(G^{-1}\circ F)(x)\leq x+R;
$$
the first inequality above was done in Lemma \ref{lem:cdfbound}, and the second inequality is proven in the same way. So

\begin{align*}
\sup_{-2R\leq x \leq 2R} (G^{-1}\circ F)'(x)=\sup_{-2R\leq x \leq 2R}\frac{p(x)}{q((G^{-1}\circ F)(x))}\leq\sup_{\substack{-2R\leq x \leq 2R \\ -R\leq y\leq R}}\frac{p(x)}{q(x+y)}=\left(\inf_{\substack{-2R\leq x \leq 2R \\ -R\leq y\leq R}}\frac{q(x+y)}{p(x)}\right)^{-1}.
\end{align*}
For convenience, let $S=\{(x,y):-2R\leq x \leq 2R, -R\leq y\leq R\}$. Now
\begin{align*}
\inf_{(x,y)\in S}\frac{q(x+y)}{p(x)}=\inf_{(x,y)\in S}\frac{1}{p(x)}\int_{-R}^{R}p(x+y-s)\,d\mu(s).
\end{align*}
Since $p$ has no local minima, the minimum value of the above integrand occurs at either $s=R$ or $s=-R$. Without loss of generality, we assume the minimum is achieved at $s=R$ (otherwise, we can replace $(x,y)$ with $(-x,-y)$ by symmetry of $S$ and $p$). So
\begin{align*}
\inf_{(x,y)\in S}\frac{q(x+y)}{p(x)}\geq\inf_{(x,y)\in S}\frac{1}{p(x)}\cdot p(x+y+R).
\end{align*}
Elementary calculus shows that the above infimum is equal to $e^{-12R^2}$ (achieved at $x=2R, y=R$). Therefore
$$
\sup_{-2R\leq x \leq 2R} (G^{-1}\circ F)'(x)=\left(\inf_{(x,y)\in S}\frac{q(x+y)}{p(x)}\right)^{-1}\leq e^{12R^2}.
$$

The case $x\leq -2R$ is dealt with in the same way as the case $x\geq 2R$, the analagous statements being:
$$
\exp\left(-2R^2-2R-\frac{1}{8}\right)p(x)\leq q(x+K(x))\leq e^{-R}\,p(x),
$$
$$
K'(x)\leq R, 
$$
$$
x+K(x)\leq(G^{-1}\circ F)(x)\leq x+R,
$$
and $q$ is increasing for $x\leq -2R$. The upper bound for $(G^{-1}\circ F)'(x)$ obtained in this case is the same as the one in the case $x\geq 2R$.

We therefore have
\begin{align*}
||G^{-1}\circ F||_{\mathrm{Lip}}\leq \max\left(\exp\left(2R^2+2R+\frac{1}{8}\right),e^{12R^2}\right)
\end{align*}
So by Proposition \ref{prop:lip}, $\mu*\gamma_1$ satisfies a LSI with constant $c(1)$ satisfying
$$
c(1)\leq 2||G^{-1}\circ F||_{\mathrm{Lip}}^2\leq \max\left(2\exp\left(4R^2+4R+\frac{1}{4}\right),2\,e^{24R^2}\right).
$$
This proves the theorem for the case $\delta=1$.\\

To establish the theorem for a general $\delta>0$, first observe that
$$
\mu*\gamma_\delta=(h_{\sqrt\delta})_*\left(((h_{1/\sqrt\delta})_*\mu)*\gamma_1\right),
$$ 
where $h_\lambda$ denotes the scaling map with factor $\lambda$, i.e., $h_\lambda(x)=\lambda\, x$. Now $(h_{1/\sqrt\delta})_*\mu$ is supported in $[-R/\sqrt\delta,R/\sqrt\delta]$, so by the case $\delta=1$ just proven, $((h_{1/\sqrt\delta})_*\mu)*\gamma_1$ satisfies a LSI with constant
$$
\max\left(2\exp\left(4(R/\sqrt\delta)^2+4(R/\sqrt\delta)+\frac{1}{4}\right),2\,e^{24(R/\sqrt\delta)^2}\right).
$$
Finally, since $||h_{\sqrt\delta}||_{\mathrm{Lip}}^2=\delta$, we have by Proposition \ref{prop:lip},
$$
c(\delta)\leq\max\left(2\delta\exp\left(\frac{4R^2}{\delta}+\frac{4R}{\sqrt\delta}+\frac{1}{4}\right),2\delta\exp\left(\frac{24 R^2}{\delta}\right)\right).
$$

In particular, when $\delta\leq 16R^2$, we have
$$
2\delta\exp\left(\frac{4R^2}{\delta}+\frac{4R}{\sqrt\delta}+\frac{1}{4}\right)\leq 2\delta\exp\left(\frac{24 R^2}{\delta}\right)
$$
so the above bound on $c(\delta)$ simplifies to
$$
c(\delta)\leq2\delta\exp\left(\frac{24 R^2}{\delta}\right).
$$
\end{proof}

\section*{Acknowledgements}
The author would like to thank his Ph.D. advisor, Todd Kemp, for his valuable insights and discussions regarding this topic.

\end{document}